\documentclass[10pt]{article}
\usepackage{amsmath, amscd, amsthm} 
\usepackage{amssymb, amsfonts}
\input xy \xyoption{all}
\CompileMatrices

\usepackage{color}

\def\qed{\hfill $\Box$ \medskip \par}
\def\Box{\framebox[10pt]{\rule{0pt}{3pt}}}

\ifx\ThmNum\undefined 
   \newtheorem{theorem}{Theorem}[section]
   
\else 
   \newtheorem{theorem}{Theorem} 
   
\fi

\newcommand{\uw}{{\underline{{w}}}}
\renewcommand{\L}{\mathbb{L}}

\newcommand{\Al}{\mathcal{A}_n({\kappa})}
\newcommand{\Sn}{\Sigma_n}

\newcommand{\Z}{\mathbb{Z}}

\newcommand{\E}{\mathcal{R}}

\renewcommand{\S}{\mathfrak{S}}
\newcommand{\LG}{\mathfrak{L}}
\newcommand{\KLG}{\mathfrak{L}^K}

\newtheorem{remark}[theorem]{Remark}
\newtheorem{proposition}[theorem]{Proposition}
\newtheorem{lemma}[theorem]{Lemma}
\newtheorem{definition}[theorem]{Definition}
\newtheorem{corollary}[theorem]{Corollary}
\newtheorem{example}[theorem]{Example}

\newcommand{\kk}{{{\mathsf{k}}}}
\newcommand{\Xo}{{\mathring{X}}}
\newcommand{\Xt}{{\widetilde{X}}}
\newcommand{\Zt}{{\widetilde{Z}}}

\renewenvironment{proof}{{\flushleft \it Proof.}}{\hfill $\square$ \vspace{2mm}}
\DeclareMathOperator{\supp}{Supp}
\DeclareMathOperator{\GL}{GL}
\DeclareMathOperator{\id}{Id}
\DeclareMathOperator{\Gr}{Gr}
\newcommand{\scal}[1]{\langle #1 \rangle}
\newcommand{\la}{{{\lambda}}}
\newcommand{\p}{{{\mathbb{P}}}}
\renewcommand{\P}{{{\mathcal{P}}}}

\title{Smooth Schubert
varieties and generalized Schubert polynomials in algebraic
cobordism of Grassmannians}
\author{Jens Hornbostel and Nicolas Perrin\footnote{The second author was supported by a public grant as part of the
Investissement d'avenir project, reference ANR-11-LABX-0056-LMH,
LabEx LMH.}}
\date{\today}

\begin{document}


\maketitle

\begin{abstract}
We provide several ingredients towards a generalization
of the Little\-wood-Richardson rule
from Chow groups to algebraic cobordism. In particular,
we prove a simple product-formula for multiplying classes of smooth Schubert 
varieties with any Bott-Samelson class in algebraic cobordism of the 
grassmannian. We also establish some results for generalized Schubert 
polynomials for hyperbolic formal group laws.
\end{abstract}

\section{Introduction}

Throughout the article, we fix a base field $\kk$  with ${\rm char}(\kk)=0$. Recall that for $G$ a reductive group over $\kk$ and $P$ a parabolic subgroup of $G$, there exist Borel type presentations of the algebraic cobordism $\Omega^*(G/P)$ of the homogeneous space $G/P$, see \cite{HK,HM}. In general, see \cite{LM}, \cite{LP} for the foundations on $\Omega^*(X)$ for a smooth projective variety $X$ over $\kk$. 

In this article, we adapt an alternative more geometric point of view.
Namely, it is known that the basis of any of these cobordism rings may be described via geometric generators, using resolutions of Schubert varieties, see below. Schubert calculus consists in multiplying these basis elements. One of the new features when passing from Chow groups to cobordism is the need of resolving the singularities of Schubert varieties. There are therefore many possible bases since a given basis element depends on the choice of a resolution of a Schubert variety. In this paper we shall mostly consider Bott-Samelson resolutions. Let us mention that some formulas for multiplication with divisor classes are already avalaible, see \cite{CPZ}, \cite{HK} and that in recent preprints of Hudson and Matsumura \cite{HM}, \cite{HM2}, Giambelli-type formulas are obtained for special classes and for a group $G$ of type $A$. There are several other recent preprints on related questions by Lenart-Zainoulline and others, see e.g. \cite{LZ}.

We also focus on groups of type $A$. In the first part we consider classes of smooth Schubert varieties in grassmannians and prove a formula for multiplying the class of a smooth Schubert variety with the class of any Bott-Samelson resolution. Several years ago, Buch achieved a beautiful generalization of the classical Littlewood-Richardson rule \cite{Bu} for $K$-theory instead of Chow groups, building on previous work of Lascoux-Schutzenberger, Fomin-Kirilov and others. In the language of formal group laws (FGL), Buch has generalized the Littlewood-Richardson rule from the additive FGL to the multiplicative FGL. In the second part, we analyse the work of Fomin and Kirilov \cite{FK1} and \cite{FK2} used by Buch, and generalize parts of it to other formal group
laws. One might hope that ultimately this will be part of a Littlewood-Richardson rule for the universal case, that is a complete Schubert calculus for
algebraic cobordism of Grassmannians.

\medskip 

Recall \cite{LM} that algebraic cobordism is the universal oriented algebraic cohomology theory on smooth varieties over $\kk$. Its coefficient ring is the Lazard ring $\L$, see \cite{La}. For any homogeneous space $X = G/P$ with $G$ reductive and $P$ a parabolic subgroup of $G$, we have a cellular decomposition of $X$ given by the $B$-orbits ($B \subset P$ a Borel subgroup of $G$) called Schubert cells and denoted by $(\Xo_w)_{w \in W^P}$ where $W^P$ is a subset of the Weyl group $W$. Choosing resolutions $\Xt_w \to X_w$ of the closures $X_w$ of $\Xo_w$ defines an additive basis of $\Omega^*(X)$ (see \cite[Theorem 2.5]{HK}). Schubert calculus aims at understanding the product in terms of these basis elements.

Write $X = \Gr(k,n)$ for the Grassmannian variety of $k$-dimensional linear subspaces in $\kk^{n}$. This is a homogeneous space of the form $G/P$ with $G = \GL_n(\kk)$ and $P$ a maximal parabolic subgroup of $G$. In the first part of the article, we compute some simple product formulas in $\Omega^*(X)$. For grassmannians, there is another indexing set for Schubert cells and their closures in terms of partitions, and we shall use this notation in the grassmannian case. In the following statement, $\la$ is a partition associated to a Schubert variety $X_\la$, that is the closure of the Schubert cell $\Xo_\lambda$ (see Subsection \ref{subsec:bruhat}). Recall also that for the grassmannian $X$, all Bott-Samelson resolutions of the Schubert variety $X_\lambda$ are isomorphic over $X$. We denote by $\Xt_\la$ this unique Bott-Samelson resolution. Finally, recall that any smooth Schubert variety in $X$ is of the form $X_{b^a}$ with $b^a$ the partition with $a$ parts of size $b$. 

Before stating the main result of section 2, recall the definition of the dual partition (see Section \ref{subsec:bruhat} for more details): for a partition $\lambda$ contained in the $k \times (n - k)$ rectangle $R$, then $\lambda^\vee$ denotes the dual partition obtained by taking the complement of $\lambda$ in $R$. For a partition $\mu$ in the $b \times a$ rectangle, the dual partition $\mu^{\vee}$ is in the $b \times a$ rectangle.

\begin{theorem}[Corollary \ref{coro-cobor}]
Let $\la \in \P(k,n)$. Then in $\Omega^*(X)$, we have
$$[X_{b^a}] \cdot [\Xt_\la] =  
\left\{\begin{array}{ll}
[\Xt_ {(\la^\vee)^{\vee_Z}}] & \textrm{for $\la \geq (b^a)^\vee$}, \\
0 & \textrm{for $\la \not \geq (b^a)^\vee$}. \\
\end{array}\right.$$
\end{theorem}

Note that for Chow groups or for $K$-theory, the above results are well known and follow from the Pieri formulae (see for example \cite{Ma} for the Chow group case and \cite{Bu} for $K$-theory, by which we always mean $K_0$). 

Note also that there are other natural resolutions of Schubert varieties considered in the litterature such as Zelevinsky's resolution \cite{Ze}. We believe that for those resolutions (which contain as a special case the resolutions considered in the cobordism Giambelli formulae of Hudson and Matsumura \cite{HM}) similar formulas should exist for the multiplication with the class of smooth Schubert varieties.

\medskip

In the second part (sections 3 and 4), inspired by Buch's method for giving a Littlewood-Richardson rule for $K$-theory, we have a closer look at generalized Schubert polynomials for cobordism. Let us recall first that for the full flag variety $X = G/B$ with $G = \GL_n(\kk)$ and $B$ a Borel subgroup, there is a Borel-type presentation of the cobordism ring (see \cite[Theorem 1.1]{HK}):

\begin{theorem}
We have an isomorphism
$\Omega^*(X) \simeq \L[x_1,\ldots,x_n]/S,$
where $\deg(x_i) = 1$ for all $i \in [1,n]$ and $S$ is the ideal generated by homogeneous symmetric polynomials of positive degree.
\end{theorem}

In particular, given a Schubert variety $X_w$ and a Bott-Samelson resolution $\Xt_\uw \to X_w$ we may write the class $[\Xt_\uw] \in \Omega^*(X)$ as a polynomial $\LG_\uw$ in the $(x_i)_{i \in [1,n]}$. In \cite{FK1,FK2}, Fomin and Kirillov give a very nice description of such polynomials for the $K$-theory case. The work of Buch \cite{Bu} builds on these results. In section 3, we compare the generalized Schubert polynomials for cobordism with those for $K$-theory (called {\em Grothendieck polynomials}), see Corollary \ref{differ}. For this, we have to restrict to hyperbolic formal group laws, that is to elliptic cohomology. Choosing a suitable generalization of the Hecke algebra, we are also able to generalize the main theorem of \cite{FK1} from $K$-theory to elliptic cohomology, see Theorem \ref{thmFK1hyp}.

\medskip

In the last section, we combine techniques and results from sections
2 and 3 to compute some explicit generalized Schubert polynomials.
In particular, we show that some of the smooth Schubert varieties
satisfy a certain symmetry, see Corollary \ref{symmetry}. For 
generalized Schubert polynomials associated to other cells, this is no longer
true already when looking at $\Gr(2,4)$, see Proposition \ref{gr24}.

\medskip

We have tried to present these two parts in a way that they can be read
essentially independently of each other. However, we emphasize that they both are 
partial solutions to the quest of a Schubert calculus for
arbitrary orientable cohomology theories. Both parts reflect 
that for general formal group laws with operators not satisfying 
the naive braid relation,
Schubert cells will lead to different elements in the corresponding
generalized cohomology theory. On the geometric side, we have different 
resolutions of a given Schubert cell, and on the combinatorial side we have 
different reduced words for a given permutation.
We hope that forthcoming work will combine these two aspects,
leading to a better understanding of general Schubert calculus.

\medskip

\section{Product with smooth Schubert varieties}
\label{section1}

\subsection{Notation}
\label{subsec:bruhat}

Let $X = \Gr(k,n)$ be the grassmannian of $k$ dimensional subspaces in $E = \kk^n$. Denote by $(e_i)_{i \in [1,n]}$ the canonical basis of $\kk^n$. 
Denote by $B$ the subgroup of upper-triangular matrices in $\GL_n(\kk)$ by $B^-$ the subgroup of lower-triangular matrices and by $T = B \cap B^-$ the subgroup of diagonal matrices. For any subset $I \subset [1,n]$ write $E_I$ for the span $\scal{e_i \ | \ i \in I}$. Set $E_i = E_{[1,i]}$ and $E^i = E_{[n+1-i,n]}$ for $i \in [1,n]$.

Call partition any non increasing sequence $\la = (\la_i)_{i \geq 1}$ of integers. The length of the partition is $\ell(\la) = \max\{i \ | \ \la_i \neq 0 \}$. For $\la$ of length $k$, we identify $\la$ with its first $k$ parts \emph{i.e.} with $(\la_i)_{i \in [1,k]}$. The weight of $\la$ is $| \la | = \sum_i \la_i$. We will also use the pictural description via Young diagrams which are left aligned arrays of $| \la |$ boxes with $\la_i$ boxes on the $i$-th line for all $i \geq 1$. A partition $\la$ fits in the $k \times (n-k)$ rectangle if its Young diagram does or equivalently if $\ell(\la) \leq k$ and $\la_1 \leq n - k$. Denote by $\P(k,n)$ the set of partitions fitting in the $k \times (n-k)$ rectangle. For $\la \in \P(k,n)$ denote by $\la^\vee  \in \P(k,n)$ its dual partition defined by $\la_i^\vee = n - k - \la_{k + 1 - i}$ for $i \in [1,k]$. We have $| \la^\vee | =  k(n-k) - | \la |$. Define $\la \leq \mu$ if $\la_i \leq \mu_i$ for all $i$.

\medskip

Recall the Bruhat decomposition: the $B$-orbits $(\Xo_\la)_{\la \in \P(k,n)}$ form a cellular decomposition. The same result holds for the $B^-$-orbits  $(\Xo^\la)_{\la \in \P(k,n)}$. Indeed these orbits are isomorphic to affine spaces: $\Xo_w \simeq \mathbb{A}_\kk^{| \la |}$ and $\Xo^\la \simeq \mathbb{A}_\kk^{\dim X - | \la |}$ as can easily be deduced from their explicit descriptions:
$$\begin{array}{l}
\Xo_\la = \{ V_k \in X \ | \ \dim(V_k \cap E_{i + \lambda_{k+1-i}}) = i \textrm{ for all $i \in [1,k]$}  \} \textrm{ and } \\
\Xo^{\la} = \{ V_k \in X \ | \ \dim(V_k \cap E^{i + n - k - \lambda_{i}}) = i \textrm{ for all $i \in [1,k]$}  \}.
\end{array}$$
Note that with this definition we have $\Xo^{\la^\vee} = w_X \cdot \Xo_\la$ where $w_X$ is the matrix permutation associated to the permutation $i \mapsto n + 1 - i$ of $[1,n]$. 
Denote by $X_\la$ the closure of $\Xo_\la$ and by $X^\la$ the closure of $\Xo^\la$. We have
$$\begin{array}{l}
X_\la = \{ V_k \in X \ | \ \dim(V_k \cap E_{i + \lambda_{k+1-i}}) \geq i \textrm{ for all $i \in [1,k]$}  \} \textrm{ and } \\
X^{\la} = \{ V_k \in X \ | \ \dim(V_k \cap E^{i + n - k - \lambda_{i}}) \geq i \textrm{ for all $i \in [1,k]$}  \}.
\end{array}$$
Inclusion induces the order on partitions: $X_\la \subset X_\mu \Leftrightarrow \la \leq \mu$. 

\begin{remark}
The bases $([X_\la])_{\la \in \P(k,n)}$ and $([X^\la])_{\la \in \P(k,n)}$ are dual bases in $CH^*(X)$ (see \cite[Proposition 3.2.7]{Ma}).
Since $X^{\la^\vee} = w_X \cdot X_\la$ we see that $([X_\la])_{\la \in \P(k,n)}$ and $([X_{\la^\vee}])_{\la \in \P(k,n)}$ are aloso dual bases. 
Note that this is no longer true in $K$-theory.
\end{remark}

\subsection{Smooth Schubert varieties, Bott-Samelson resolution and cobordism}
\label{subsec:smooth}

The smooth Schubert varieties in $X$ are sub-grassmannians (see for example 
\cite[Theorem 6.4.2]{LB} 
or \cite[Theorem 1.1]{GR}, and \cite{brion-polo} or \cite{gorenstein} for more details on the singular locus and the type of singularities).
The partitions corresponding to these smooth Schubert varieties are of the form $\la = (\la_1,\cdots,\la_k)$ with $\la_i = b$ for $i \in [1,a]$ and $\la_i = 0$ for $i > a$ for some integers $a \in [1,k]$ and $b \in [1,n-k]$. Denote this permutation by $\la = b^a$. As a variety we have
$$\begin{array}{l}
X_{b^a} = \{ V_k \in X \ | \ E_{k - a} \subset V_k \subset E_{k + b} \} \textrm{ and } \\
X^{{b^a}^\vee} = \{ V_k \in X \ | \ E^{k - a} \subset V_k \subset E^{k + b} \}.
\end{array}$$
Moreover we have $X_{{b^a}} \simeq \Gr(a,a+b) \simeq X^{{b^a}^\vee}$.

\medskip

As already mentioned,  Schubert varieties are in general singular. There exist several resolutions of singularities. We recall here the Bott-Samelson resolutions of Schubert varieties which were first introduced by Bott and Samelson \cite{BS} as well as Hansen \cite{H} and Demazure \cite{D} for full flag
varieties. These constructions and their properties 
carry over easily to partial flags $G/P=\Gr(k,n)$. See e.g.
\cite{Fu}, \cite{LB} for more details.
We give here an explicit description of these resolutions in the spirit of configuration spaces (see \cite{MA} or \cite{small-reso}). Note also that for Schubert varieties in $X$, these resolutions are \emph{canonical} in the sense that they do not depend on the choice of a reduced expression. 

For a partition $\la$ and a pair of integers $(i,j)$ write $(i,j) \in \la$ if $i \in [1,k]$ and $j \in [1,\la_i]$ and $(i,j) \not\in \la$ else. Define $V_{(i,j)} = E_{k + j - i}$ for all $(i,j) \not\in \la$ where $E_i$ is the zero space for $i \leq 0$ and $E_i = \kk^n = E_n$ for $i \geq n$. Define $Y_\la = \prod_{(i,j) \in \la} \Gr(k + j - i , n)$. Set 
$$\Xt_\la = \{ (V_{(i,j)})_{(i,j) \in \la} \in Y_\la \ | \ V_{(i+1,j)} \subset V_{(i,j)} \subset V_{(i,j+1)} \textrm{ for all $(i,j) \in \lambda$} \}.$$
The projection $\pi_\la : \Xt_\la \to X$ defined by $\pi_\la((V_{(i,j)})_{(i,j) \in \la}) = V_{1,1}$ induces a birational morphisms onto $X_\la$. Furthermore, one easily checks that $\Xt_\la$ has 
the structure of a tower of $\p^1$-bundles so that $\Xt_\la$ is smooth. 
  The morphisms $\pi_{\la} : \Xt_{\la} \to X_{\la}$ are called the Bott-Samelson resolutions of $X_\la$.

These resolutions define classes $[\pi_\la : \Xt_\la \to X]$ in the cobordism $\Omega^*(X)$ of $X$. We write $[\Xt_\la]$ for these classes. The classes $([\Xt_\la])_{\la \in \P(k,n)}$ form a basis in any oriented cohomology theory and especially in cobordism:
$$\Omega^*(X) = \bigoplus_{\la \in \P(k,n)} \L[\Xt_\la],$$
where $\L$ is the Lazard ring, see \cite{HK}.

\subsection{Products in cobordism}

We want to understand the product in $\Omega^*(X)$ with the classes $[X_{b^a}]$. Note that the class $[X_{b^a}]$ is well defined without considering any resolution since $X^{b^a} \simeq \Gr(a,a+b)$ is smooth, hence its cobordism class is well defined. 

\subsubsection{Sub-grassmannians}

Let $Z = \Gr(a,a+b)$ be the grassmannian of $a$-dimensional vector subspaces of $\kk^{a+b}$. Let $(f_i)_{i \in [1,a+b]}$ be the canonical basis of $\kk^{a+b}$. Define $F_i = \scal{f_j \ | \ j \in [1,i]}$ and $F^i = \scal{f_j \ | \ j \in [a+b+1-i,a+b]}$. For $\la \in \P(a,a+b)$ a partition contained in the $a \times b$ rectangle define the Schubert variety in $Z$ (as above in $X$):
$$\begin{array}{l}
Z_\la = \{ V_a \in Z \ | \ \dim(V_a \cap F_{i + \la_{a+1-i}}) \geq i \textrm{ for all $i \in [1,a]$}  \} \textrm{ and } \\
Z^\la = \{ V_a \in Z \ | \ \dim(V_a \cap F^{i + b - \la_{i}}) \geq i \textrm{ for all $i \in [1,a]$}  \}.
\end{array}$$
If $w_Z : \kk^{a+b} \to \kk^{a+b}$ is the endomorphism defined by $f_i \mapsto f_{a + b + 1 -i}$, then $Z^\la = w_Z \cdot Z_{\la^{\vee_Z}}$ with $\mu = \la^{\vee_Z}$ defined by $\mu_i = b - \la_{a + 1 -i}$ for all $i \in [1,a]$. 

Now define Bott-Samelson resolutions in $Z$. Define $W_{(i,j)} = F_{a + j - i}$ for all $(i,j) \not\in \la$ where $F_i$ is the zero space for $i \leq 0$ and $F_i = \kk^{a+b} = F_{a+b}$ for $i \geq a+b$. 
Define $A_\la = \prod_{(i,j) \in \la} \Gr(a + j - i , a+b)$. Set
$$\Zt_\la = \{ (W_{(i,j)})_{(i,j) \in \la} \in A_\la \ | \ W_{(i+1,j)} \subset W_{(i,j)} \subset W_{(i,j+1)} \textrm{ for all $(i,j) \in \lambda$} \}.$$
The projection $\pi_{\la}^Z : \Zt_\la \to Z$ defined by $\pi_{\la}^Z((W_{(i,j)})_{(i,j) \in \la}) = W_{1,1}$ induces a birational morphism onto $Z_\la$.

Embed $Z$ in $X$ with image $X_{(b^a)}$ as follows. Let $u : \kk^{a+b} \to \kk^n$ be the linear map defined by $u(f_i) = e_{k - a + i}$ for all $i \in [1,a+b]$. Note that $u(\kk^{a+b}) = E_{[k - a + 1 , k + b]}$. Denote by $v : Z \to X$ the closed embedding defined by $W_a \mapsto E_{k-a} \oplus u(W_a)$.

Embed $Z$ in $X$ with image $X^{(b^a)^\vee}$ as follows. Let $u' : \kk^{a+b} \to \kk^n$ be the linear map defined by $u'(f_i) = e_{n - k - b + i}$ for all $i \in [1,a+b]$. Note that $u'(\kk^{a+b}) = E_{[n - k - b + 1 , n - k + a]}$. Denote by $v' : Z \to X$ the closed embedding defined by $W_a \mapsto E^{k-a} \oplus u'(W_a)$.

\subsubsection{Intersection with Schubert varieties}

In this subsection we consider the classes of closed subvarieties $Y \subset X$ in Chow groups or in $K$-theory. To avoid introducing more notation we denote both theses classes by $[Y]$ and specify in which theory we are working. The product with the class $[X_{b^a}]$ in Chow groups or for $K$-theory is easy to compute. 

\begin{lemma}
Let $\la \in \P(k,n)$. We have
$$v(Z) \cap X^\la = X_{b^a} \cap X^\la =  
\left\{\begin{array}{ll}
\emptyset & \textrm{for $\la \not \leq b^a$}, \\
v(Z^\la) & \textrm{for $\la \leq b^a$}.\\
\end{array}\right.$$
\end{lemma}

\begin{proof}
Let $\mu = b^a$. As is well known, the intersection $X_{\mu} \cap X^\la$ is non empty if and only if $\la \leq \mu$. Assume this holds we also know that $X_{\mu} \cap X^\la$ is a Richardson variety thus reduced, irreducible of dimension $| \mu | - | \la |$. Since $Z^\la$ has dimension $| \mu | - | \la |$ it is enough to prove the inclusion $v(Z^\la) \subset X_{b^a} \cap X^\la$. By construction, we have $v(Z) = X_{b^a}$ thus $v(Z^\la) \subset X_{b^a}$. We prove the inclusion $v(Z^\la) \in X^\la$. Recall the definition
$$X^{\la} = \{ V_k \in X \ | \ \dim(V_k \cap E^{i + n - k - \lambda_{i}}) \geq i \textrm{ for all $i \in [1,k]$}  \}.$$
Since $\lambda$ is contained in the $a \times b$ rectangle we have $\ell(\la) \leq a$ thus the conditions $\dim(V_k \cap E^{i + n - k - \lambda_{i}}) \geq i$ for $i>a$ become $\dim(V_k \cap E^{i + n - k}) \geq i$ and are trivially satisfied. We need to check the conditions  $\dim(V_k \cap E^{i + n - k - \lambda_{i}}) \geq i$ for $i \in [1,a]$ and $V_k = v(W_a)$ with $W_a \in Z^\la$. For all $i \in [1,a]$, we have $\dim(V_a \cap F^{i + b - \lambda_{i}}) \geq i$. Applying $v$ we get the inequality $ \dim(v(V_a) \cap v(F^{i + b - \lambda_{i}}) \cap E_{[k-a+1,k+b]}) \geq i$. But $v(F^{i + b - \lambda_{i}} \cap E_{[k-a+1,k+b]}) = E_{[k + 1 - \la_i -i , k + b]} \subset E_{[k + 1 - \la_i -i , n]}  = E^{i + n - k - \lambda_{i}}$. In particular $\dim(v(V_a) \cap E^{i + n - k - \lambda_{i}}) \geq i$ for $i \in [1,a]$ proving the result.
\end{proof}

Remark that $v(w_Z(F^i)) = E_{k-a} \oplus u(F_i) = E_i$, thus for $\la \in \P(a,a+b)$, we have $v(Z_\la) = X_\la$. 
In particular, we have $v(w_Z \cdot Z^\la) = v(Z_{\la^{\vee_Z}}) = X_{\la^{\vee_Z}}$. Consider $\kk^{a+b}$ as a subspace of $\kk^n$ via the embedding $u$ and let $w^Z$ be the endomorphism of $\kk^n$ obtained by extending $w_Z$ with the identity on the complement $\scal{e_i \ | \ i \not\in [k-a,k+b]}$. We have $w^Z \circ v = v \circ w_Z$.

\begin{corollary}
\label{cor:inter}
Let $\la \in \P(a,a+b)$. We have
$$v(Z) \cap X^\la = X_{b^a} \cap X^\la =  
\left\{\begin{array}{ll}
\emptyset & \textrm{for $\la \not \leq b^a$}, \\
w^Z \cdot X_{\la^{\vee_Z}} & \textrm{for $\la \leq b^a$}.\\
\end{array}\right.$$
\end{corollary}

\begin{corollary}
\label{cor:inter-dual}
Let $\la \in \P(a,a+b)$. We have
$$X_\la \cap v'(Z) = X_\la \cap X^{{b^a}^\vee} =  
\left\{\begin{array}{ll}
\emptyset & \textrm{for $\la \not \geq (b^a)^\vee$}, \\
w_X w^Z \cdot X_{(\la^\vee)^{\vee_Z}} & \textrm{for $\la \geq (b^a)^\vee$}.\\
\end{array}\right.$$
\end{corollary}

\begin{proof}
Set $\mu = \la^\vee$, apply Corollary \ref{cor:inter} to $\mu$ and multiply with $w_X$.
\end{proof}

\begin{corollary}
\label{coro:coho}
Let $\la \in \P(a,a+b)$. In $CH^*(X)$, we have
$$[X_\la] \cup [X_{{b^a}}] =  
\left\{\begin{array}{ll}
[X_{(\la^\vee)^{\vee_Z}}] & \textrm{for $\la \geq (b^a)^\vee$},\\
0 & \textrm{for $\la \not \geq (b^a)^\vee$}. \\
\end{array}\right.$$
\end{corollary}

\begin{remark}
The same result holds for $K$-theory, see \cite{Bu}.
\end{remark}

Our aim is to generalise the above results to Bott-Samelson resolutions and  to cobordism. For this, the dual point of view of Corollary \ref{cor:inter-dual} is better suited.

\subsubsection{Fiber product}

Let $\mu$ be a partition in the $a \times b$ rectangle and let $\mu' = (\mu^{\vee_Z})^\vee$. We construct an embedding of $\Zt_\mu \to \Xt_{\mu'}$. We denote by $v' : \Gr(i,a+b) \to \Gr(i+k-a,n)$ the embeddings induced by $u'$ as follows: $v'(W_i) = u'(W_i) \oplus E^{k-a}$.

First remark that $\mu \leq \mu'$ and that we get $\mu'$ from $\mu$ by adding $k - a$ lines (with $n-k$ boxes) and $n - k - b$ columns (with $k$ boxes). In other words $\mu'_i = n-k$ for $i \in [1,k-a]$ and $\mu'_i = \mu_i + n - k - b$ for $i \in [k - a +1 , k]$.

Let $(W_{(i,j)})_{(i,j) \in \mu} \in \Zt_\mu$. We define $(V_{(i,j)})_{(i,j) \in \mu'}$ as follows.

\noindent
For $i \in [1,k-a]$ and $j \in [1,n-k-b]$, set
$$V_{(i,j)} = (v'(W_{(1,1)}) \oplus E_{j-1}) \cap E_{n +1 -i}$$
For $i \in [k-a+1,k]$ and $j \in [1,n-k-b]$
$$V_{(i,j)} =  (v'(W_{(i -(k-a),1)}) \oplus E_{j-1}) \cap E_{n + a - k}$$
For $i \in [1,k-a]$ and $j \in [n-k-b+1,n-k]$
$$V_{(i,j)} =  (v'(W_{(1,j-(n-k-b))}) \oplus E_{n-k-b}) \cap E_{n +1 -i}$$
For $i \in [1,k-a]$ and $j \in [1,n-k-b]$
$$V_{(i,j)} =  (v'(W_{(i -(k-a),j-(n-k-b))}) \oplus E_{n-k-b}) \cap E_{n +a -k}.$$
For $(i,j) \not\in \mu'$ we set $V_{(i,j)} = (v'(W_{(i -(k-a),j-(n-k-b))}) \oplus E_{n-k-b}) \cap E_{n +a -k}$.

\begin{lemma}
We have $(V_{(i,j)})_{(i,j) \in {\mu'}} \in \Xt_{\mu'}$.
\end{lemma}

\begin{proof}
Recall that $u'(\kk^{a+b}) = E_{n-k-b,n-k+a}$, that $E^{k-a} \subset v'(W)$ and that $v'(W) \subset E^{k+b}$ for any subspace $W \subset \kk^{a+b}$. In particular, in the above definition all sums are direct and all intersections are transverse. This implies $\dim V_{(i,j)} = k + j -i$ thus $(V_{(i,j)})_{(i,j) \in {\la'}} \in Y_{\mu'}$. For $(i,j) \not\in \mu'$ we have $V_{(i,j)} = (v'(W_{(i -(k-a),j-(n-k-b))}) \oplus E_{n-k-b}) \cap E_{n + a -k} = E_{k+j-i}$.
An easy check proves $V_{(i+1,j)} \subset V_{(i,j)}  \subset V_{(i,j+1)}$. The result follows.
\end{proof}

\begin{lemma}
The map $\varphi : \Zt_\mu \to \Xt_{\mu'}$ is a closed embedding.
\end{lemma}

\begin{proof}
We have $u'(W_{(i,j)}) = V_{(i + k-a,j+n-k-b)}) \cap E^{k+b}$. Since $u$ is injective, the result follows.
\end{proof}

\begin{lemma}
The map $\psi: \Zt_\mu \to X$ defined by $(W_{(i,j)})_{(i,j) \in \mu} \mapsto V_{(1,1)}$ factors through $v'(Z)$.
\end{lemma}

\begin{proof}
We have $V_{(1,1)} = v'(W_{(1,1)}) = u'(W_{(1,1)}) \oplus E^{k-a}$. In particular $E^{k-a} \subset V_{(1,1)} \subset E^{k+b}$. The result follows.
\end{proof}

\begin{proposition}
Let $\mu \in \P(a,a+b)$ and consider $\Zt_\mu$ as an $X$-scheme via $\psi$. We have $\Xt_{\mu'} \times_X v'(Z) = \Xt_{\mu'} \times_X X^{(b^a)^\vee} \simeq \Zt_\mu$.
\end{proposition}

\begin{proof}
We have morphisms $\varphi : \Zt_\mu \to \Xt_{\mu'}$ and $\psi : \Zt_\mu \to v'(Z)$ with $\varphi$ a closed embedding. Furthermore the map $\pi_{\mu'} : \Xt_{\mu'} \to X$ is given by $(V_{(i,j)})_{(i,j) \in {\mu'}} \mapsto V_{(1,1)}$ so the composition $\pi_{\mu'} \circ \varphi$ is the map $\psi$. In particular we have a morphism $\varphi \times \psi : \Zt_\mu \to \Xt_{\mu'} \times_X v'(Z)$. This is a closed embedding since $\varphi$ is a closed embedding). To prove that this is an isomorphism is it enough to prove that $\Xt_{\mu'} \times_X v'(Z)$ is irreducible and smooth of dimension $| \mu | = \dim \Zt_\mu$. But $v'(Z) = X^{(b^a)^\vee}$ and $\Xt_{\mu'}$ are in general position. By Kleimann-Bertini \cite{bertini} any irreducible component is of dimension $| \mu | - \textrm{codim}_X v'(Z) = | \mu |$. By Bertini again, the fiber product of $v'(Z)$ with the locus in $\Xt_{\mu'}$ where $\pi_{\mu'}$ is not an isomorphism has dimension strictly less than $| \mu |$ and is therefore never an irreducible component. Now since $v'(Z) \cap X_{\mu'}$ is irreducible, the same holds for $\Xt_{\mu'} \times_X v'(Z)$. Furthermore by Bertini again this fiber product is smooth and therefore reduced.
\end{proof}

\begin{corollary}
Let $\la \in \P(k,n)$. As $X$-schemes, we have
$$\Xt_\la \times_X v'(Z) = \Xt_\la \times_X X^{b^a}  \simeq  
\left\{\begin{array}{ll}
\emptyset & \textrm{for $\la \not \geq (b^a)^\vee$}, \\
\Zt_\mu & \textrm{for $\la \geq (b^a)^\vee$},\\
\end{array}\right.$$
with $\mu = (\la^\vee)^{\vee_Z}$ for $\la \geq (b^a)^\vee$ and $\Zt_\mu$ is considered as an $X$-scheme via $\psi$.
\end{corollary}

\subsubsection{Cobordism}

We construct another $X$-scheme isomorphism between $\Zt_\mu$ and $w_X w^Z \cdot \Xt_\mu$. Here $\Zt_\mu$ is an $X$-scheme via $\psi$ while $w_X w^Z \cdot \Xt_\mu$ is an $X$-scheme via $w_Xw^Z \circ \pi_\mu$. The action of $w_X$ and $w^Z$ on $\Xt_\mu$ being defined via the embedding of $\Xt_\mu$ in $Y_\mu$ and the action on the later is given by the diagonal action on each factor (recall that $Y_\mu$ is a product of grassmannians $\Gr(i,n)$ on which $w_X$ and $w^Z$ act).

Let $(W_{(i,j)})_{(i,j) \in \mu} \in \Zt_\mu$. We define $(V_{(i,j)})_{(i,j) \in \mu}$ as follows. For $(i,j) \in \mu$, set $V_{(i,j)} = v'(W_{(i,j)})$. For $(i,j) \not\in \mu$, set $V_{(i,j)} = w_Xw^Z \cdot E_{k + j - i}$

\begin{lemma}
We have $(V_{(i,j)})_{(i,j) \in \mu} \in w_X w^Z \cdot \Xt_\mu$.
\end{lemma}

\begin{proof}
For $(i,j)$, $(i+1,j)$ and $(i,j+1)$ in $\mu$, the conditions $V_{(i+1,j)} \subset V_{(i,j)} \subset V_{i,j+1)}$ are clearly satisfied. We only need to check these conditions for $(i+1,j)$ or $(i,j+1)$ not in $\mu$. But $(i,j) \not \in \mu$, we have $W_{(i,j)} = F_{a+j-i}$ thus $v'(W_{(i,j)}) = v'(F_{a+j-i}) = E^{k-a} \oplus E_{[n-k-b+1,n-k-b+a+j-i]} = w_Xw^Z \cdot E_{k+j-i} = V_{(i,j)}$ and the result follows.
\end{proof}

\begin{proposition}
Let $\mu \in \P(a,a+b)$. The $X$-schemes $\Zt_\mu$ (via $\psi$) and $w_Xw^Z \cdot \Xt_\mu$ are isomorphic.
\end{proposition}

\begin{proof}
The above morphism sending $(W_{(i,j)})_{(i,j) \in \mu} \in \Zt_\mu$ to $(V_{(i,j)})_{(i,j) \in \mu} \in \Xt_\mu$ is a closed embedding. Since both schemes are smooth are irreducible of the same dimension, this map is an isomorphism. Wee need to check that the morphisms to $X$ coincide. But the composition $\Zt_\mu \to w_Xw^Z \cdot \Xt_\mu \to X$ is given by $(W_{(i,j)})_{(i,j) \in \mu} \mapsto (V_{(i,j)})_{(i,j) \in \mu} \mapsto V_{(1,1)}$ and therefore maps $(W_{(i,j)})_{(i,j) \in \mu} \in \Zt_\mu$ to $v'(W_{(1,1)}) = \psi(W_{(1,1)})$. It coincides with $\psi$.
\end{proof}

\begin{corollary}
Let $\la \in \P(k,n)$. As $X$-schemes, we have
$$\Xt_\la \times_X v'(Z) = \Xt_\la \times_X X^{b^a}  \simeq  
\left\{\begin{array}{ll}
\emptyset & \textrm{for $\la \not \geq (b^a)^\vee$}, \\
w_Xw^Z \cdot \Xt_ {(\la^\vee)^{\vee_Z}}& \textrm{for $\la \geq (b^a)^\vee$}.\\
\end{array}\right.$$
\end{corollary}

\begin{corollary}
\label{coro-cobor}
Let $\la \in \P(k,n)$. Then in $\Omega^*(X)$, we have
$$[X_{b^a}] \cdot [\Xt_\la] =  
\left\{\begin{array}{ll}
[\Xt_ {(\la^\vee)^{\vee_Z}}] & \textrm{for $\la \geq (b^a)^\vee$}, \\
0 & \textrm{for $\la \not \geq (b^a)^\vee$}. \\
\end{array}\right.$$
\end{corollary}

\begin{remark}
1. These results were inspired by several similar results in other cohomology theories. In particular, the results explained in Corollary \ref{coro:coho} are the classical part of Seidel symmetries in quantum cohomology \cite{seidel}. The results of Seidel are not explicit but were made explicit in \cite{CMP2} and \cite{CMP4}. These results extend to quantum $K$-theory. This will be presented in a forthcoming work \cite{BCMP}. We expect the same results to be valid in quantum cobordism once defined.

2. We expect more general results of the same type for other homogeneous space. 
These will be studied by the second author in forthcoming work.
\end{remark}

\section{Generalized Schubert polynomials and generalized Hecke algebras}

In this section, we discuss how far classical Grothendieck polynomials, which are representatives of Schubert classes in Borel's presentation of $K$-theory, are from the representatives in Borel's presentation of algebraic cobordism of Bott-Samelson resolutions of Schubert varieties. For $K$-theory (that is $K_0$), the computation of polynomial representatives for classes of Schubert varieties has been done by Fomin-Kirillov \cite{FK1}, \cite{FK2}.
We establish a generalization of
the main theorem of \cite{FK1}. Building on their work, Buch \cite{Bu} computed Littlewood-Richard\-son rules for $K$-theory.

\subsection{Divided difference operators}

Recall that $K$-theory corresponds to the multiplicative formal group law. The methods of Buch and Fomin-Kirilov do not generalize to the universal formal group law, that is algebraic cobordism. However, we will show that they apply in a much weaker form to {\em hyperbolic formal group laws} (see Definition \ref{defhfgl} below). For $i \in [1,n]$, let $s_i$ be the transposition of $[1,n]$ exchanging $i$ and $i+1$. 

\begin{definition}\label{generalDDO} Let $F$ be a
formal group law over $R$ with inverse $\chi$.

1. For $i \in [1,n]$, define $\sigma_i \in {\rm End}(R[[x_1,\ldots,x_n]])$ by 
$$(\sigma_i f)(x_1,\ldots,x_n)=f(x_{s_i(1)},\ldots,x_{s_i(n)}).$$

2. For $i \in [1,n]$, define $C_i,\Delta_i \in {\rm End}(R[[x_1,\ldots,x_n]])$ by 
$$C_{i}=(\id + \sigma_i)\frac1{F(x_{i},\chi(x_{i+1}))} \textrm{ and } \Delta_{i}=\frac{1}{F(x_{i+1},\chi(x_i))}(\id - \sigma_{i}).$$
\end{definition}

\begin{remark}
Note that the above operators are well defined in $R[[x_1,\ldots,x_n]]$ since $F(x,\chi(y))$ can be written $(x-y)g(x,y)$ with $g \in R[[x,y]]$.
\end{remark}

This definition is taken from \cite[p.71]{HK} and \cite[section 3]{CPZ}. When applying it to the additive formal group law, one recovers the usual definition as e.g. in \cite[section 2.3.1]{Ma} up to a sign (observe that $\sigma_i \circ F(x_{i+1},\chi(x_i)) = F(x_{i},\chi(x_{i+1}))$). For the multiplicative formal group
law $F(x,y)=x + y + \beta xy$, the definition of $C_{i}$ yields the $\beta$-DDO $\pi_i^{(\beta)}$ of
\cite{FK1}, which for $\beta=-1$ specializes to the isobaric DDO of
\cite{Bu}. Moreover, still for the multiplicative formal group law $F(x,y)=x + y + \beta xy$, the operator $\Delta_{i}$ above (which equals the one of \cite[section 3]{CPZ}) coincides up to sign with the operator
$\pi_i^{(\beta)} + \beta$ which appears in \cite[Lemma 2.5]{FK1}.

\medskip

Recall \cite{BE1} that the braid relations for the operators $C_{i}$
only hold in if the FGL is additive or multiplicative. We therefore need to keep track of reduced expressions to define generalized Schubert polynomials, which is not necessary in \cite[Definition 2.1]{FK1}.

\subsection{Generalized Schubert polynomials}

The following definition generalizes both Schubert polynomials
for Chow groups and Grothendieck polynomials for $K$-theory.

\begin{definition}\label{defLG}
Let $w$ be a permutation and $\uw$ be a reduced expression of $w$ as product in the $(s_i)_{i \in [1,n-1]}$. Define the generalized Schubert polynomial $\LG_\uw$ by induction:
\begin{itemize}
  \item[(a)] $\LG_{1}(x_1,\ldots,x_{n}) = x_1^{n-1}x_{2}^{n-2}\ldots x_{n-1}$
  \item[(b)] $\LG_{\uw s_i} := C_{i}\LG_{\uw}$ if $\uw s_i$ is a reduced
word.
\end{itemize}
\end{definition}

Note that this notation is different from the one
used in \cite{FK1} and elsewhere: Our $\LG_{1}$
corresponds to their $\LG_{w_0}$ and our $\LG_{\uw}$ to their $\LG_{w_0w}$. We decided to adopt this notation since there is a unique class for the point as well as a unique reduced expression for $1$, but a Bott-Samelson resolution and a polynomial $\LG_{\uw_0}$ for each reduced expression $\uw_0$ of the element $w_0$.

For any permutation $w$, the Bott-Samelson resolutions $\Xt_\uw \to X_w$ of the Schubert variety $X_w$ are indexed by the reduced word $\uw$ of $w$. It was proved in \cite[Theorem 3.2]{HK} that the polynomial $\LG_\uw$ represents the class of the resolution $\Xt_\uw \to X_w$ in $\Omega^*(G/B)$. 

\medskip

Let $S$ be the ideal in $R[[x_1,\ldots,x_n]]$ generated by symmetric 
polynomials of positive degree. The polynomial
$\LG_ {1}$ corresponds to the cobordism class of a point.
Modulo $S$, the polynomial $n!\LG_ {1}$ has several equivalent
descriptions (compare e.g. \cite[Remark 2.7]{HK}, where 
$\Delta_n$ differs by a scalar from the one below).

\begin{lemma}\label{pt} Let $A^*(-)$ be an oriented cohomology theory
with associated FGL $F$. Then we have
$$(a) \ \Delta_n:=\prod_{1 \leq i < j \leq n}(x_i - x_j) \equiv n! x_1^{n-1}x_{2}^{n-2} \cdots x_{n-1}
 = n! \LG_{1} \ mod \ S$$
and
$$(b) \ \Delta_n \equiv \Delta_n^F:=\prod_{1 \leq i < j \leq n}(x_i -_{F} x_j)
\ mod \ S.$$  
\end{lemma}

\begin{proof} 
To show (a), one first verifies that modulo $S$ we have
$\prod_{1 < i \leq n}(x_1 - x_i) \equiv n x_1^{n-1}$,
deriving the equality $\prod_{1 \leq i \leq n}(x - x_i) \equiv x^{n}$
and setting $x=x_1$.
Then one shows $x_1^{n-1}p(x_1,\ldots,x_{n-1}) \equiv 0$
for any symmetric non-constant
polynomial $p(x_1,\ldots,x_{n-1})$, writing
$p(x_1,\ldots,x_{n-1})=x_1 q(x_1,\ldots,x_{n})$ 
and using that $x_1^n \equiv 0$ modulo $S$.
Now proceed by induction on $n$. The claim holds for $n=1$.
Using the factorization $$\prod_{1 \leq i < j \leq n}(x_i - x_j)=\prod_{1 < i < j \leq n}(x_i - x_j) \prod_{1 < i \leq n}(x_1 - x_i),$$ the claim for $n$ 
follows using the induction hypothesis for $n-1$ and the above 
two equalities modulo $S$.

For (b), note that $x_i -_{F} x_j = 0$ if $x_j=x_i$, which implies that
 $x_i -_{F} x_j$ is divisible by $x_i - x_j$,
Hence $x_i -_{F} x_j = (x_i - x_j)a(x_i,x_j)$ 
with $a(x_i,x_j)=1 + b(x_i,x_j)$ and $b \in (x_i,x_j)$. Thus
$\Delta_n^F = \Delta_n + \Delta_n q(x_1,\ldots,x_n)$
with $q(0,\ldots,0)=0$. Now using part (a) and the 
equality $x_1^n \equiv 0 \ mod \ S$,
we deduce that $\Delta_n x_i \equiv 0 \ mod \ S$
for $i=1$ and thus (use a suitable permutation) 
for all $i$. Hence $\Delta_n q(x_1,\ldots,x_n) \equiv 0 \ mod \ S$ 
as claimed.
\end{proof}

\begin{remark}
  Some authors use $x_n^{n-1} x_{n-1}^{n-2} \cdots x_2$ in place of $x_1^{n-1}x_2^{n-2} \cdots x_{n-1}$. Modulo $S$ these two classes only differ by the
sign $(-1)^{\frac{n(n-1)}{2}}$.
\end{remark}

\subsection{Hyperbolic formal group laws}

We now define hyperbolic formal group laws, which generalize
the additive and multiplicative ones. 

\begin{definition}\label{defhfgl}
The hyperbolic formal group law $F$ over $R=\Z[\mu_1,\mu_2]$ and its inverse $\chi$ are given by
$$F(x,y)=\frac{x+y - \mu_1xy}{1 + \mu_2 xy} \textrm{ and }\chi(x)=-\frac{x}{1 - \mu_1x}.$$
\end{definition}

Recall that formal group laws are by definition
power series in two variables, and all fractions here and below
may be written as such. See e.g. \cite{BB}, \cite[Example 2.2 (d)]{HMSZ}
and \cite[2.2]{LZ} for more on hyperbolic formal group laws.
Combining their computations, we see that
$$F(x,y)=x + y - \mu_1xy + \mu_2(x^2y+xy^2) + \mu_2 \mu_1 x^2 y^2 + O(5).$$
In section 4.2. below, we explain how these FGLs lead to
certain elliptic cohomology theories $E^*(-)$.
If $\mu_2=0$, 
these cohomology theories specialize to Chow groups
(if $\mu_1=0$), $K_0$ (if $\mu_1$ is invertible,
thus sometimes called periodic $K$-theory), 
connective $K_0$ and (if $\mu_1=0$ but $\mu_2 \neq 0$)
theories associated with Lorentz FGLs.
\medskip

\begin{definition}\label{defkappa}
Let $F$ be a formal group law. Define 
$$\kappa_i=\kappa_i^F= \frac{1}{F(x_i,\chi(x_{i+1}))} +  \frac{1}{F(x_{i+1},\chi(x_i))}.$$
\end{definition}

\begin{remark}
In the above definition, $\kappa_i$ is a formal series. Indeed, writing $F(x,\chi(y)) = (x-y)g(x,y)$ with $g$ a formal series with constant term equal to $1$, we get
$$\kappa_i = \frac{g(y,x) - g(x,y)}{(x-y)g(x,y)g(y,x)}$$
Since the numerator vanishes for $x = y$ there exists a formal series $h$ such that $g(y,x) - g(x,y) = (x-y)h(x,y)$ and we get
$$\kappa_i = \frac{h(x,y)}{g(x,y)g(y,x)}$$
which can be written as a formal series.
\end{remark}

\begin{remark}\label{easycomp}
An easy computation shows that $\Delta_{i}=\kappa_i-C_{i}$.
\end{remark}

\begin{example}\label{kappa}
The three formal group laws we have studied so far are $F_a$, $F_m$ and 
$F_e$, namely the additive, the multiplicative and the elliptic (or hyperbolic) 
formal group laws:
$$F_a(x,y) = x + y, \ \ \ F_m(x,y) = x + y - \mu_1 xy \ \ \ \textrm{and}  \ \ \ F_e(x,y) = \frac{x + y - \mu_1 xy}{1 + \mu_2 xy}.$$
In these cases, we have $\kappa_i^{F_a}= 0$, 
$\kappa_i^{F_m}= \kappa_i^{F_e}= \mu_1$. So in all these examples, 
$\kappa:=\kappa_i$ is independent of $i$.
\end{example}

We now define a variant of the Hecke algebra generalizing \cite[Definition 2.2]{FK1} with respect to a fixed hyperbolic formal group law $F$. Setting $\mu_2=0$, we obtain the Hecke algebra of \cite{FK1}, corresponding to (connective or periodic) $K$-theory.

\begin{definition}\label{generalHecke} For any hyperbolic
formal group law $F$ defined over $R=\Z[\mu_1,\mu_2]$ consider the commutative ring  $\E:=R[[x_1,\ldots,x_n]]/(\mu_1-\kappa)$.
The {\it generalized Hecke algebra} $\Al$ is the quotient
of the associative $\E$-algebra $\E \langle u_1,\ldots,u_n \rangle$
by the relations
\begin{itemize}
\item $u_i x_j=x_j u_i \textrm{ for all $i,j$, }$
\item $u_iu_j = u_ju_i \textrm{ for $|i-j| > 1$, }$
\item $u_i u_{i+1} u_i = u_{i+1} u_i u_{i+1}  \textrm{ for all $i$,}$
\item $u_i^2 = - \mu_1 u_i \textrm{ for all $i$, }$
\item $\mu_2x_i x_{i+1}u_i=0 \textrm{ for all $i$.}$ 
\end{itemize}
\end{definition}

Although this algebra generalizes the ones of \cite{FK1}, \cite{Bu} and others, note that it is different from the formal {\em Demazure} algebras studied in \cite{CPZ}, \cite{HMSZ}.  See Remark \ref{different} below for more details on this.

\begin{remark}
Note that the elements $u_i$ satisfy the braid relations. Hence for any permutation $w$, we can define the element $u_w$  as $u_w = u_{i_1} \ldots u_{i_r}$, where $w = {s_{i_1}} \ldots s_{s_{i_r}}$ is any reduced expression of $w$.
\end{remark}

We now generalize \cite[Theorem 2.3]{FK1} from multiplicative to hyperbolic formal group laws. Define
$$\S(x_1,\ldots,x_{n-1})=\prod_{j=1}^{n-1} \prod_{i=n-1}^j(1+x_ju_i).$$ 

\begin{theorem}\label{thmFK1hyp}
For any hyperbolic FGL, in the generalized Hecke algebra $\Al$ of Definition \ref{generalHecke}, we have
$$\S(x_1,\ldots,x_{n-1})=\sum_{w \in \Sn}  \LG_{\uw} u_{w_0w},$$
where $\uw$ is any reduced expression of $w$ and $w_0(i)=n+1-i$
as usual.
\end{theorem}

Before proving this theorem, we compare the generalized Schubert polynomials 
$\LG_{\uw}$ with the corresponding Grothendieck polynomials for $K$-theory.

\begin{definition}
Let $w$ be a permutation and $w = {s_{i_1}} \ldots s_{s_{i_r}}$ be any reduced expression. 

1. The support of $w$ is the set $\supp(w) = \{i_1,\ldots,i_r\}$. This is independent of the chosen reduced expression since its is preserved by the braid relations.

2. Define $I(w)$ as the ideal in $\E$ generated by the polynomials $\mu_2 x_i x_{i+1}$ for $i \in \supp(w_0w)$.

3. Let $\KLG_w$ be the $K$-theoretic Grothendieck polynomial representing $X_w$.
\end{definition}

\begin{corollary}\label{differ}
  Let $\uw = s_{\alpha_{i_1}} \ldots s_{\alpha_{i_r}}$ be a reduced expression of $w$. Then for $w$ a permutation and $\uw$ be any reduced expression for $w$, in
$\E$ we have
  $$\LG_{\uw} = \KLG_{w} \textrm{ mod } I(w).$$
\end{corollary}

Some parts of the proof of Theorem \cite[Theorem 2.3]{FK1} are formal and immediately generalize to arbitrary formal group laws. Lemma 2.5 of \cite{FK1} just rephrases Remark \ref{easycomp}.
Several other crucial parts of the proof do not generalize to arbitrary FGLs. However, they do generalize to hyperbolic FGLs when working with the generalized Hecke algebra $\Al$. An important point in choosing hyperbolic FGL is the fact that the $\kappa_i$ are independent of $i$, so we have an action of the symmetric group on $\Al$ given by permutation on the variables $x_i$. From now on, we fix a hyperbolic formal group law $F$ and a positive integer $n$.

\begin{lemma}\label{FKhypkey}
Set 
$\alpha_i(x) = (1+xu_{n-1}) \cdots (1+xu_i).$
Then we have the following equalities in $\Al$.

1. $\alpha_{i+1}(x_{i+1})=\alpha_i(x_{i+1})(1 + \chi(x_{i+1})u_i)$;

2. $1 + \chi(x_i)u_i =(1 + F(x_{i+1},\chi(x_i))u_i)(1+\chi(x_{i+1})u_i)$;

3. $\Delta_i(1 + \chi(x_{i+1})u_i)=(1+\chi(x_{i+1})u_i)u_i$.

\end{lemma}

\begin{proof} 
1. The equality $\alpha_{i+1}(x_{i+1})(1 + x_{i+1}u_i)=\alpha_i(x_{i+1})$
implies $\alpha_{i+1}(x_{i+1})(1 + x_{i+1}u_i)(1 + \chi({x_{i+1}})u_i)=
\alpha_i(x_{i+1})(1 + \chi({x_{i+1}})u_i)$. A straightforward computation
shows that $(1 + x_{i+1}u_i)(1 + \chi({x_{i+1}})u_i)=1$.  

2. To prove the claim, it suffices to prove that
$$(F(x_{i+1},\chi(x_i)) + \chi(x_{i+1}) - \chi(x_i))u_i + \chi(x_{i+1})F(x_{i+1},\chi(x_i))u_i^2=0,$$
or equivalently that
$$(F(x_{i+1},\chi(x_i)) + \chi(x_{i+1}) - \chi(x_i)- \mu_1 \chi(x_{i+1})F(x_{i+1},\chi(x_i)))u_i =0$$
This holds by a computation using the explicit formulas for
$F$ and $\chi$ and the relation $\mu_2x_i x_{i+1}(x_i - x_{i+1}) u_i=0$. We use a stronger relation in the definition of our Hecke algebra since we need $x_i - x_{i+1}$ to be a non zero divisor for the next computation.

3. We have
$$\begin{array}{ll}
  \renewcommand\arraystretch{2}
  -\Delta_i(1+\chi(x_{i+1})u_i) & = \displaystyle{\frac{(1+\chi(x_{i})u_i) - (1+\chi(x_{i+1})u_i)}{F(x_{i+1},\chi(x_i))}} \\ [2em]
  & = \displaystyle{\frac{1+ F(x_{i+1},\chi(x_i))u_i) - 1}{F(x_{i+1},\chi(x_i))}(1+\chi(x_{i+1})u_i) } \\ [2em]
  & = (1 + \chi(x_{i+1})u_i)u_i. \\
\end{array}
$$ 
Here we simplified by $x_i - x_{i+1}$ and the second equality follows from part 2. 
\end{proof}

\begin{proposition}\label{someYBE}
In the above notation, for all $i$ we have the commutation
$$\alpha_i(x_i)\alpha_i(x_{i+1})=\alpha_i(x_{i+1})\alpha_i(x_i).$$
\end{proposition}

\begin{proof}
Since we have the same relations for the $u_i$ as in \cite{FK1}, 
the proof of their Lemma 2.6 generalises to our situation.
More precisely, we may apply \cite[Corollary 5.4]{FK2} as its assumptions
(see \cite[section 2]{FK2}) are satisfied
in our generalized Hecke algebra. 
\end{proof}

\noindent
\textit{Proof of Theorem \ref{thmFK1hyp}.} From $\S(x_1,\ldots,x_{n-1}) = \alpha_1(x_1) \ldots \alpha_{n-1}(x_{n-1})$ we get
$$ \S(x_1, \ldots,x_{n-1}) = \alpha_1(x_1) \ldots \alpha_i(x_{i+1})(1 + \chi(x_{i+1})u_i)\alpha_{i+2}(x_{i+2}) \ldots \alpha_{n-1}(x_{n-1}).$$
Using Lemma \ref{FKhypkey}.1, this implies that $\Delta_i(\S(x_1,\ldots,x_{n-1}))$ is equal to the following formulae:
$$\begin{array}{l}
\alpha_1(x_1)\ldots\alpha_{i-1}(x_{i-1})\Delta_i\alpha_i(x_i)\alpha_i(x_{i+1})(1 + \chi(x_{i+1})u_i)\alpha_{i+2}(x_{i+2})\ldots \alpha_{n-1}(x_{n-1}) \\
=\alpha_1(x_1)\ldots\alpha_{i-1}(x_{i-1})\alpha_i(x_i)\alpha_i(x_{i+1})
\Delta_i(1 + \chi(x_{i+1})u_i)\alpha_{i+2}(x_{i+2})\ldots \alpha_{n-1}(x_{n-1})\\
= - \alpha_1(x_1)\ldots \alpha_i(x_i) \alpha_i(x_{i+1})
(1 + \chi(x_{i+1})u_i)u_i\alpha_{i+2}(x_{i+2})\ldots \alpha_{n-1}(x_{n-1})\\
= - \alpha_1(x_1)\ldots \alpha_i(x_i)\alpha_i(x_{i+1})
(1 + \chi(x_{i+1})u_i)\alpha_{i+2}(x_{i+2})\ldots \alpha_{n-1}(x_{n-1})u_i.
\end{array}$$
Here the third last equality follows from
Proposition \ref{someYBE} 
and the second last one from Lemma \ref{FKhypkey}.3.
We thus have shown $$-\Delta_i(\S(x_1,\ldots,x_{n-1})=(\S(x_1,\ldots,x_{n-1})) u_i$$
which corresponds precisely to the induction step in Definition
\ref{defLG}, using that $\Delta_{i}= \kappa - C_{i}$
and $u_i^2 = - \kappa u_i$. 
More precisely, write $\S= \sum \hat{\LG}_{w} u_{w_0w}$,
where the sum taken over all $w \in \Sn$. We wish to show that $\hat{\LG}_{w} u_{w_0w} = {\LG}_{\uw} u_{w_0w}$ by a ascending induction on the length of $w$. 
For $w = 1$ the claim is obviously true. Now fix $w \neq 1$ and choose $i$ such that $ws_i$ is reduced. Consider the coefficient of $u_{w_0w}$ in
$$(C_i - \kappa_i)\S= - \Delta_i \S = \S u_i.$$
Using that $u_i^2= - \kappa_i u_i$ and the fact that $w_0ws_i < w_0w$, we deduce that
$$(C_i - \kappa_i) \hat{\LG}_{w} u_{w_0w} = (\hat{\LG}_{ws_i} - \kappa_i \hat{\LG}_{w}) u_{w_0w},$$
hence
$C_i \hat{\LG}_{w}u_{w_0w} =  \hat{\LG}_{w s_i} u_{w_0w}$ as required.
\qed

\begin{remark}
  \label{different}
  Note that the computations from \cite{FK1} cannot be done in the formal Demazure
 algebra of \cite{HMSZ}. E.g., the equality
$$(1 + x_{i+1}u_i)(1 + \chi({x_{i+1}})u_i) = 1$$
which was used to prove Lemma \ref{FKhypkey} above
does not hold, not even for the additive FGL. 
This is related to the failure of
$\kappa_i \Delta_i=\Delta_i \kappa_i$.
\end{remark}

As for hyperbolic formal group laws $\kappa_i$ is independent of $i$
(see Example \ref{kappa}), several other parts in Buch's
article \cite{Bu} on the Littlewood-Richardson rule for $K_0$ easily 
generalize to hyperbolic formal group laws when working with the generalized
Hecke algebra $\Al$ of Definition \ref{generalHecke}. For example, similar to \cite[p. 41]{Bu}, it is possible to introduce
a stable generalized Schubert polynomial $colim \ \LG_{1^m \times \uw}$ of $\LG_{\uw}$
and trying to analyze its behaviour along the lines of \cite[section 6]{FK2}.
Also, there is a well-defined analogon $\LG_{\nu/\lambda}$ of the polynomial
$G_{\nu/\lambda}$ which is crucial for \cite[Theorem 3.1]{Bu}, as the construction
in \cite[p. 41/42]{Bu} provides a reduced word ${\uw}$ rather than just a permutation $w$. 
However, for hyperbolic formal group laws the operators
$C_i$ do no longer satisfy the classical braid relation but
a twisted version of it, namely $C_i C_{i+1} C_i + \mu_2 C_i =C_{i+1} C_i C_{i+1} + \mu_2 C_{i+1}$ 
\cite{HMSZ}. This will lead to additional difficulties 
when arguing inductively using these $C_i$ and the corresponding
geometric operators as e.g. in \cite[section 8]{Bu}. This is
also related to the discussion in \cite[section 6]{LZ}.
On the other hand, Proposition \ref{someYBE} is wrong
already for small values of $n$ and $i$ when replacing
the classical braid relation for the $u_i$ by its
twisted analog in the definition of $\Al$.
We hope to return to these questions in future work. 


\section{Some examples}

\subsection{Polynomials representing some smooth Schubert varieties}

We first compute generalized Schubert polynomials for some of the smooth Schubert varieties considered in Section \ref{section1}. Let $X = \Gr(k,n)$ be a grassmannian and let $\lambda$ be a partition of the form $b^a$. Denote by $\mathfrak{G}_\lambda$ the polynomial in $\Omega^*(G/B) \simeq \L[x_1,\cdots,x_n]/S$ representing the pull-back along the canonical quotient map
$\pi:G/B \to X$ of the cobordism class $[X_\la \to X]$. 
Recall \cite[section 3.2.4]{HML} that the induced map $\pi^*:\Omega^*(\Gr(k,n)) \to
\Omega^*(G/B)$ is a ring monomorphism which identifies
$\Omega^*(\Gr(k,n))$ with an explicit subring of $\L[x_1,\cdots,x_n]/S$.
The results in the sequel may thus be stated in either of these rings.  
(Recall that there is a standard map, see e.g. \cite[p.42]{Bu}, from
partitions to permutations with corresponds to $\pi^*$ and geometric 
operators for $K$-theory and Chow groups.)

\begin{proposition}
In $\Omega^*(G/B)$, we have the formulas 
$$\mathfrak{G}_{(n-k)^a} = (x_{k+1} \cdots x_n)^{k - a} \textrm{ and } \mathfrak{G}_{b^k} = (x_{1} \cdots x_k)^{n - k - b}.$$
\end{proposition}

\begin{corollary}\label{symmetry}
  The classes of $[X_{(n-k)^a} \to X]$ and $[X_{b^k} \to X]$ are represented by the same polynomial in any oriented cohomology theory.
\end{corollary}

\begin{proof}
  Indeed we have $[X_{(n-k)^a} \to X] = \mathfrak{G}_{(n-k)^a}$ and $[X_{b^k} \to X] = \mathfrak{G}_{b^k}$, so this is independent of the FGL.
\end{proof}

\begin{remark}
  We will see in the next subsection that this is no longer the case for the other classes of smooth Schubert varieties. Indeed, in Proposition \ref{gr24}, we prove that the class of the line in the elliptic cohomology of $\Gr(2,4)$ is given by $x_1x_2(x_1+x_2) - \mu_1 x_1^2x_2^2$ and therefore depends on the FGL.
\end{remark}

\begin{proof}
We prove the first formula. For this, we start with an easy lemma.

\begin{lemma}
In $\Omega^*(X)$, we have $[X_{(n-k)^{k-1}}]^a = [X_{(n-k)^{k-a}}]$ and $[X_{(n-k-1)^k}]^b = [X_{(n-k-b)^k}]$.
\end{lemma}

\begin{proof}
We need to prove the formula $[X_{(n-k)^{k-1}}] \cdot [X_{(n-k)^{k-a}}] = [X_{(n-k)^{k-a-1}}]$. But the first class is represented by the subgrassmannian $X_{n-k} = \{ V_k \in X \ | \ E_1 \subset V_k \}$ while the second class is represented by $X^{{(n-k)^{k-a}}^\vee} = X^{(n-k)^{a}} = \{ V_k \in X \ | \ E^{a} \subset V_k \}$. The product is represented by the intersection of these varieties and since $E_1$ and $E^{a}$ do not meet we get
$$X_{n-k} \cap  X^{(n-k)^{a}} = \{ V_k \in X \ | \ E_1 \oplus E^{a} \subset V_k \}.$$
This last variety is a translate via the action of $\GL_n(\kk)$ of 
$X_{(n-k)^{k-a-1}} = \{ V_k \in X \ | \ E_{a+1} \subset V_k \}$ proving the first formula. The second formula is obtained along the same lines or deduced from the first one using the isomorphism $\Gr(k,n) \simeq \Gr(n-k,n)$.
\end{proof}

Finally since $X_{(n-k)^{k-1}}$ 
is the zero locus of a section of the tautological quotient 
bundle whose Chern roots are $x_{k+1},\cdots,x_n$, the first equality
of the proposition follows (see for example the proof of \cite[Lemma 6.6.7]{LM}).For the second formula, we just need to remark that $X_{(k-1)^k}$ is the zero locus of a global section of the dual of the tautological subbundle and apply the same method (or use the isomorphism $\Gr(k,n) \simeq \Gr(n-k,n)$ again).
\end{proof}

\subsection{Elliptic cohomology of $\Gr(2,4)$}

In this subsection, we present explicit results concerning elliptic cohomology, \emph{i.e} for the hyperbolic FGL, of $\Gr(2,4)$. We compute the polynomial representing any Bott-Samelson class as well as their products. 

Let $X = \Gr(2,4)$ and let $\lambda$ be a partition. Denote by $\LG_\lambda$ the polynomial in $\Omega^*(G/B) \simeq \L[x_1,x_2,x_3,x_4]/S$ representing the pull-back in $G/B$ of the cobordism class $[\Xt_\la \to X]$ where $\Xt_\lambda$ is the Bott-Samelson resolution of $X_\lambda$. 

Recall the hyperbolic FGL of \cite[Example 63]{BB}
as in subsection 3.3. above. By the universal property of
the formal group law of $\Omega^*$ established in
\cite{LM}, we have a unique morphism of formal group laws,
which yields in particular a ring morphism $\L \to \Z[\mu_1, \mu_2]$.
This map is called "Krichever genus" and studied in detail
in loc. cit.. In particular, $\mu_i$ has cohomological degree
$-i$ for $I=1,2$.
Note that (unlike in the bigraded case, see e.g. \cite{LYZ})
this always yields an oriented cohomology theory, as there is no
Landweber exactness condition to check. As the theory
$E^*(-)$ is oriented in the sense of \cite{LM}, 
the analogs of the above theorems also hold for $E^*(G/B)$ 
and $E^*(\Gr(2,4))$, and the natural transformation
$\Omega^*(-) \to E^*(-)$ commutes in particular with 
the ring monomorphisms $\pi^*$.
Below, we use the notations $\Xt_{\la}$ and $\LG_{\la}$
for elements in $E^*(-)$ as well.
 
\begin{proposition}\label{gr24}
In $E^*(\Gr(2,4))$, we have the following formulas:
$$\begin{array}{l}
\LG_{(00)} = x_1^2x_2^2, \\
\LG_{(10)} = x_1x_2(x_1 + x_2) - \mu_1 x_1^2x_2^2, \\
\LG_{(20)} = x_1^2 + x_1x_2 + x_2^2 - \mu_1 x_1x_2(x_1 + x_2) - \mu_2 x_1^2x_2^2, \\
\LG_{(11)} = x_1x_2 - \mu_2 x_1^2x_2^2, \\
\LG_{(21)} = x_1 + x_2 - \mu_1 x_1x_2 - \mu_2 x_1x_2(x_1 + x_2) - \mu_1\mu_2 x_1^2x_2^2, \\
\LG_{(22)} = 1 - \mu_2 (x_1 + x_2)^2 + \mu_1^2\mu_2 x_1^2x_2^2.\\
\end{array}$$
\end{proposition}

\begin{proof}
Since the fiber of the map $\pi : G/B \to \Gr(2,4)$ is isomorphic to $\p^1 \times \p^1$, the pull-back $\pi^*[\Xt_{\la}] \in E^*(G/B)$ of a Bott-Samelson class in $\Gr(2,4)$ is again a Bott-Samelson class $X_{\uw}$. (Note that this is not true anymore in higher dimensions.) Moreover in this case, we can explicitely
write down the reduced word $\uw$ corresponding to $\lambda$

under $\pi^*$. Now we wish to compute $\LG_{\lambda} \in E^*(\Gr(2,4)) \subset E^*(G/B)$. The above together with the results of \cite{HK} imples that
both in $\Omega^*(G/B)$ and $E^*(G/B)$, we have $\pi^*[X_{(00)}] = C_1C_3(\LG_1)$, $\pi^*[X_{(10)}] = C_1C_3C_2(\LG_1)$, $\pi^*[X_{(20)}] = C_1C_3C_2C_3(\LG_1)$, $\pi^*[X_{(11)}] = C_1C_3C_2C_1(\LG_1)$, $\pi^*[X_{(21)}] = C_1C_3C_2C_1C_3(\LG_1)$ and $\pi^*[X_{(22)}] = C_1C_3C_2C_1C_3C_2(\LG_1)$. Now the results follow
from $\LG_1=x_1^3 x_2 ^2 x_3$
and explicit computations with the $C_i$ done with the help of a computer.
\end{proof}

We computed everything in elliptic cohomology for sake of simplicity, but a similar computation can be done in $\Omega^*(X)$. 

\begin{remark}
In elliptic cohomology, the mutiplication formula for the square of the hyperplane class in the Bott-Samelson basis is the same as the one in $K$-theory, namely $\LG_{(21)}^2 = \LG_{(20)} + \LG_{(11)} - \mu_1 \LG_{(10)}.$
\end{remark}

\noindent
Jens Hornbostel, 
Bergische Universit\"at Wuppertal,
Fachgruppe Mathematik und Informatik,
Gau{\ss}strasse 20, 42119 Wuppertal, Germany.

\noindent
\texttt{hornbostel@math.uni-wuppertal.de}.

\medskip

\noindent
Nicolas Perrin, Laboratoire de Math\'ematiques de Versailles, UVSQ, CNRS, Universit\'e Paris-Saclay, 78035 Versailles, France.

\noindent
\texttt{nicolas.perrin@uvsq.fr}.

\end{document}